\documentclass[12pt,twoside]{amsart}  
\usepackage{amsmath,amsthm}  
\usepackage{amssymb,amscd}  
\usepackage[dvips]{color}  
\usepackage[latin1]{inputenc}

\newtheorem{Theorem}{Theorem}[section]
\newtheorem{Lemma}[Theorem]{Lemma}%[section]
\newtheorem{Proposition}[Theorem]{Proposition}%[section]
\def\H{{\mathcal H}}
\def\O{{\mathcal O}}

\def\mod{{\rm mod}}
\def\Hom{{\rm Hom}}
\def\End{{\rm End}}
\def\Ind{{\rm Ind}}

\def\Res{{\rm Res}}
\def\Sch{{\mathcal S}}
\def\GL{{\rm GL}}
\def\Z{{\mathcal Z}}

\newcommand{\hcr}{\operatorname{^*R}}
\newcommand{\hci}{\operatorname{R}}

\begin{document}

\title{On decomposition numbers and Alvis-Curtis duality}

\author{Bernd Ackermann}   
\address{Inst.f.Algebra u.Zahlentheorie, Universität Stuttgart, 70550  
  Stuttgart, Germany}  
\email{ackermann@mathematik.uni-stuttgart.de}  
  
\author{Sibylle Schroll}  
\address{ Mathematical Institute,   
24-29 St Giles, Oxford OX1 3LB, United Kingdom}  
\email{schroll@maths.ox.ac.uk}  
\thanks{The   
second author acknowledges support through a Marie Curie    
Fellowship}  

\maketitle

\begin{abstract}
We show that for general linear groups ${\rm GL}_n(q)$ as well as for
 $q$-Schur algebras
  the knowledge of the modular Alvis-Curtis duality over fields of
characteristic $\ell$, $\ell \nmid q$, is equivalent
to the knowledge of the decomposition numbers.
\end{abstract}

\section{Introduction}

In the beginning of the nineties, Brou\'e's abelian
defect group conjecture \cite{B} related the homological and the character
theoretic aspects of representation theory. It follows from the conjecture that certain
character correspondences should in fact be consequences of derived
equivalences. A related phenomenon is the Alvis-Curtis
duality. Originally it was defined as a character duality of finite groups of
Lie type, which in particular sends an irreducible character onto
another irreducible character up to sign. Brou\'e has shown \cite{B}
that it is what he calls a
perfect isometry, which is a character correspondence
with signs and some arithmetic properties.
In '99 Cabanes and Rickard \cite{CR} then showed
that the Alvis-Curtis duality is in fact induced by a derived
equivalence obtained by tensoring with a certain
complex.

On the other hand, questions on perfect isometries naturally lead to
questions on decomposition numbers.
These numbers are virtually unknown for almost all groups.
In particular, the question of
how to calculate the decomposition matrices of
the general linear group has not yet been satisfactorily answered.
James gave an algorithm for calculating the
decomposition matrices of $\GL_{n}(q)$, for $n \leq 10$,  in non-describing
characteristic \cite{J3}. But for larger $n$ there is no algorithm for the calculation of
the decomposition matrix of $\GL_n(q)$.
Similary, except for small cases \cite{F, M1, M2, JLM, Ta}, the decomposition
matrices of the $q$-Schur and the Hecke algebra are not known. In this
paper, we show that the complete knowledge of the Alvis-Curtis duality
answers these questions for the general linear group and
the $q$-Schur algebra.

More precisely, we construct a
cochain complex for the $q$-Schur algebra. We then relate it to the
complex inducing the Alvis-Curtis duality \cite{CR} as well as to a
complex for the Hecke algebra \cite{LS}.
We show that this complex for the $q$-Schur algebra induces a self-derived
equivalence and thus gives rise to a duality operation in the
Grothendieck group. We analyse how the different duality
operators arising in this context
act on the different simple modules in the modular case. In fact, we
show that calculating this action for the
general linear group and for the $q$-Schur algebra
is equivalent to calculating the decomposition matrices of
these algebras.

In the following we describe the layout of the paper. Section 2 contains
notation and background results needed later on. In section
3, Theorem 3.2, we show the main result for general linear groups. A link with
 the Mullineux map and the Hecke algebra is established in Theorem 4.1 in
section 4. Section 5 contains the construction of the complex for the
$q$-Schur algebra. In Theorem 5.1 we show that it induces a derived equivalence
 and we calculate the induced action on the simple modules in the
Grothendieck group. Finally in Theorem 6.1 in section 6 we give a
summary of the results on
the different decomposition matrices.

\section{Preliminaries}\label{sec2}

Let $G = {\rm GL}_n(q)$ and $(K,\O,k)$ an $\ell$-modular system
with $\ell$ coprime to $q$. Let $e$ be the smallest integer $i$
such that modulo $\ell$
$$ 1 + q + \ldots q^{i-1} = 0 $$
Throughout we fix the following notation.
Denote by $T$ the maximal torus of invertible diagonal matrices in $G$,
by $U$ the group of upper unitriangular  matrices
%whose diagonal entries are $1$
and set $B = UT$.
Let $W$ be the Weyl group of $G$. Then $W \cong \mathfrak{S}_n$ and we
identify $\mathfrak{S}_n$ with the subgroup of permutation matrices of
$G$.
Denote by $S$ the generating set of $\mathfrak{S}_n$ corresponding to the set of
basic transpositions
$s_i=(i,i+1)$ for $1\leq i\leq n-1$. For any composition $\lambda$ of $n$
denote by $\mathfrak{S}_\lambda$ the associated Young subgroup of
$\mathfrak{S}_n$ and by  $P_\lambda$ the standard parabolic subgroup
of $G$ generated by $B$ and $\mathfrak{S}_\lambda$. Denote by $U_\lambda$
the unipotent radical of $P_\lambda$
and by $L_\lambda$ the standard Levi complement of $U_\lambda$ in
$P_\lambda$.
%Fix the
%convention $\mathfrak{S}_\emptyset = 1$, $U_\emptyset = U$, $P_\emptyset = B$
%and $L_\emptyset = T$.
For any subgroup  $H$ of $G$ let $e_H$ be the
idempotent of $KG$ defined by $e_H = \frac{1}{|H|}\underset{x\in
  H}\sum\ x$.
Note that if $H$ is an $\ell'$-subgroup then $e_H \in RG$
for $R$ any of $\{K,\O ,k\}$.
Denote by ${\mathcal E}(G,1)$ the unipotent character series, that is the
irreducible constituents of $\Ind^G_B(K)$. Define the central idempotent
 $f= \sum_{\chi} e(\chi)$
 of $KG$ where
$\chi$ runs over the set of unipotent characters of $G$ and
where $e(\chi) = \frac{\chi(1)}{|G|}\sum_{g \in G}
\chi(g^{-1})g$. Then $f$ is the block idempotent of the sum of all
unipotent $\ell$-blocks.
In a similar way
define $f_\lambda$ as the sum of the $e(\psi)$ where $\psi$ runs through the
unipotent characters of $L_\lambda$.
Let $G_{ss}$ be the set of semisimple elements of $G$ and
$G^{reg}_{ss}$ the set of $\ell$-regular semisimple elements  of
$G$. By $(g) \in_G G$ we denote the conjugacy class of $g$ in $G$.

\bigskip

For $R \in \{K,\O ,k\}$ and any composition $\lambda$ of $n$ denote
by $R^G_{L_\lambda}$ the Harish-Chandra
induction functor that associates to a $RL_\lambda$-module $N$ the $RG$-module
$RG e_{U_\lambda} \otimes_{RL_\lambda} N$. Its left and right adjoint called
Harish-Chandra restriction and denoted by $\hcr^G_{L_\lambda}$ associates to
a $RG$-module $M$ the $RL_\lambda$-module $e_{U_\lambda}M$.

\bigskip

\subsection{Combinatorics}
Denote  by
$\bar{\Lambda}(k,n)$ the set of compositions of $n$ with
exactly $k$ non-zero parts and denote by
$\bar{\Lambda}^+(k,n)$ the subset containing compositions
$\lambda=(\lambda_1, \ldots, \lambda_k)$ such that
$\lambda_1 \geq \ldots \geq \lambda_k$. Let
 $\Lambda(k,n)= \bigoplus_{i \leq k} \bar{\Lambda}(i,n)$ and
$\Lambda^+(k,n)= \bigoplus_{i \leq k} \bar{\Lambda}^+(i,n)$ .
Write $\Lambda(n)$ for $\Lambda(n,n)$  and similarly
$\Lambda^+(n)$ for $\Lambda^+(n,n)$.
For $\nu= (\nu_1, \ldots, \nu_a) \in \Lambda(n)$
define $\Lambda(\nu)$ as the subset of
all compositions $\gamma=(\gamma_1, \ldots, \gamma_{h})$
$\in \Lambda(n)$
such that $(\gamma_1, \ldots, \gamma_{h_1}) \in \Lambda(\nu_1)$,
$(\gamma_{h_1+1}, \ldots, \gamma_{h_1+h_2}) \in \Lambda(\nu_2)$ and so on
until
$(\gamma_{h_1+ \ldots +h_{a-1}+1}, \ldots, \gamma_{h})
\in \Lambda(\nu_a)$.
Write $|\lambda|$ for the number of non-zero
parts of a composition $\lambda$.
An $e$-regular partition of $n$ is a partition $\lambda = (\lambda_1 ,
\ldots , \lambda_a)$ such that if $\lambda_{i+1} =  \lambda_{i+2} =
\ldots =\lambda_{i+h}$ then $h <e$.

There exists a partial order on the set of compositions of $n$ induced by the
partial order on the generating set $S$ of $\mathfrak{S}_n$ given by inclusion of subsets. More precisely,
the subsets of $S$ are in bijection with compositions of $n$~: Let
$\lambda= (\lambda_1, \ldots, \lambda_h ) $ be a composition of
 $n$ and set
$$\lambda_i^+= \lambda_1 + \lambda_2 + \ldots + \lambda_i$$
for $1 \leq i \leq h$.
Then $\lambda$ corresponds to the subset $I_\lambda$ of $S$ given by
$I_\lambda = \{s_1, \ldots, s_n \} \setminus \{\lambda_i^+ | 1 \leq i
\leq h\}$. Note that $|I_\lambda|=(n-1)-(h-1)=n-h$.

An easy calculation then shows

\bigskip

\begin{Lemma} Suppose $I_\lambda$ and $I_\mu$ are subsets of $S$
  corresponding to $\lambda \in
\bar{\Lambda}(h,n)$ and $\mu \in \bar{\Lambda}(k,n)$. Then

(i) $|I_\lambda| \leq |I_\mu|$ if and only if $h \geq k$

(ii) $I_\lambda \subset I_\mu$ if and only if for all $\mu_i^+$ there exists $\lambda_j^+$ such
that  $\mu_i^+ = \lambda_j^+$.
\end{Lemma}

\bigskip

For compositions satisfying $(ii)$ we write $\lambda \preceq
\mu$. Note that this is a partial
order and that $\lambda \preceq \mu$  implies $\lambda\unlhd\mu$ in the
dominance order.

\bigskip

\subsection{Simple modules of the general linear group}\label{sec22}
Let $s$ be a semisimple element of $G$. Then the rational canonical form
 of $s$ is some block diagonal matrix where the diagonal blocks are
 given 
% diag$((\sigma_1)^{k_1}, \ldots,
%(\sigma_r)^{k_r})$ where $(\sigma_i)$ is
as companion matrices $(\sigma_i)$ of
the elementary divisors $s_i$ of $s$ which are the irreducible factors
of the minimal polynomial of $s$ over $\operatorname{GF}(q)$ since $s$
is semisimple. Assume that $(\sigma_i)$ appears precisely $k_i$ times
on the diagonal and let $d_i$ be the degree of $\sigma_i$, then
$(\sigma_i)$ is a $d_i\times d_i$-matrix and $d_1k_1 + \dots + d_rk_r = n$.
The centralizer $C_G(s)$ of $s$ in $G$ is then isomorphic to
$$  C_G(s) \cong {\rm GL}_{k_1}(q^{d_1}) \times \ldots \times
{\rm GL}_{k_r}(q^{d_r}).$$
Details and proofs can be found in \cite{FongSrinivasan}. With these notations define
 $\varepsilon_s = (-1)^{n + k_1 + \ldots + k_r}$.

Let
$\lambda^{(i)}$ be a partition of $k_i$.
Following James \cite{J2} we have an irreducible
$K\GL_{d_ik_i}(q)$-module $S(s_i,\lambda^{(i)})$ called Specht module
and each irreducible $KG$-module is of the form
\begin{equation*}
S(s,\tilde{\lambda})=R^G_L \left( S(s_1,\lambda^{(1)})\otimes
\dots \otimes S(s_r,\lambda^{(r)})\right)
\end{equation*}
where $\tilde{\lambda}$ is the
multipartition $(\lambda^{(1)},\lambda^{(2)},\dots,\lambda^{(r)})$ of
$(k_1, \ldots, k_r)$
and $L$ is the Levi subgroup
$\GL_{d_1k_1}(q)\times \dots \times \GL_{d_rk_r}(q)$.
Note that the ordinary irreducible character of $S(s,\tilde{\lambda})$
corresponds exactly to the character
$\chi_{s,\tilde{\lambda}}$ in the parametrization given by Green
\cite{G} and the set
$$\{S(s,\tilde{\lambda}) | (s)=(s_i^{k_i}) \in_G G_{ss} ,
\tilde{\lambda} \vdash (k_1, \ldots , k_r)  \}$$
gives a complete set of irreducible $KG$-modules \cite{J2}.

In \cite{J2} James constructs a certain $OG$-lattice
$S_O(s,\tilde{\lambda})$ in $S_K(s,\tilde\lambda)$. It is shown there
that the $\ell$-modular reduction $S_F(s,\tilde\lambda)$ of this lattice
no longer needs to be irreducible.
However, if $s$ is $\ell$-regular it has an irreducible head
$L(s,\tilde{\lambda})$ and we have
$$
L(s,\tilde{\lambda})=R^G_L \left( L(s_1,\lambda^{(1)})\otimes \dots
\otimes L(s_r,\lambda^{(r)})\right)
$$
where each $L(s_i,\lambda^{(i)})$ is the irreducible head of $S(s_i,\lambda^{(i)})$.
Moreover, the set
$$\{L(s,\tilde{\lambda}) | (s)=(s_i^{k_i}) \in_G G^{reg}_{ss} ,
\tilde{\lambda} \vdash (k_1, \ldots , k_r)  \}$$
is a complete set of irreducible $kG$-modules.

\bigskip

\subsection{Hecke algebras} Let $R \in \{K, \O, k \}$ and denote by
$\H_R$ or by  ${\H}$ the Iwahori Hecke algebra
$\H_{R,q}(\mathfrak{S}_n)$. This is defined
as the free $R$-module with basis $\{ T_w | w \in S_n\}$ and with
multiplication given by
$$
T_wT_s = \left\{
\begin{array}{ll}
T_{ws} & \mbox{if} \;\;\ \ell(ws) > \ell(w) \\
qT_w + (q-1) T_{ws} & \mbox{if} \;\; \ell(ws) < \ell(w) \\
\end{array} \right.
$$
for $w \in \mathfrak{S}_n$ and $s \in S$. For every composition $\lambda$ of $n$
we denote by $\H_\lambda$ the parabolic subalgebra of $\H$ isomorphic
to the Iwahori Hecke algebra
$\H_{R,q}(\mathfrak{S}_\lambda)$.

For every parabolic subalgebra $\H_\lambda$ there is an induction
functor
$$\begin{array}{ccccc} \Ind_{\H_\lambda}^\H &:& \H_\lambda - \mod &
  \rightarrow & \H - \mod \\
&& N & \mapsto & \H \otimes_{\H_\lambda} N \\
\end{array}$$
This is an exact functor and its left and right adjoint is given by
$$\begin{array}{ccccc} \Res_{\H_\lambda}^\H &:& \H - \mod &
  \rightarrow & \H_\lambda - \mod \\
&& M & \mapsto & \H_\lambda \otimes_{\H} N \\
\end{array}$$

A full set of irreducible $\H_K$-modules is parametrised as follows
$$\{S^\lambda | \lambda \in \Lambda^+(n) \}.$$
As $\H_k$-modules these need no longer be irreducible; however, if the
partition $\lambda$ is $e$-regular, then  $S^\lambda$ has an
irreducible head $L^\lambda$ and the set
\begin{equation*}
\{ L^\lambda | \lambda \in \Lambda^+(n),\ \text{$\lambda$ is
  $e$-regular}\}
\end{equation*}
is a complete set of irreducible $\H_k$-modules \cite{DJ1}.

In \cite{I}
an involution $\alpha: \H \rightarrow \H$ is defined by
 $T_w \mapsto (-q)^{\ell(w)} (T_{w^{-1}})^{-1}$ for $w \in \mathfrak{S}_n$.
For an $\H$-module $V$ define another $\H$-module $_\alpha V$ by setting
$h.v = \alpha(h)v$ for all $h \in \H$ and $v \in {_\alpha V}$. Then if $q=1$,
$_\alpha V = sign \otimes V$, where $sign$ is the sign representation of
$\mathfrak{S}_n$. Furthermore
it is shown in \cite{Br} that for any $e$-regular
partition $\lambda$  of $n$ and any simple $\H$-module $L^\lambda$
we have $_\alpha (L^\lambda) =
L^{m(\lambda)}$ where $m(\lambda)$ is the image of $\lambda$
by the Mullineux map $m$ (for the definition of $m$ see \cite{Mu}).

\bigskip

\subsection{$q$-Schur algebras}

For each  composition $\lambda$ of $n$ there is a permutation module
$M^\lambda$ of $\H= \H_R$ for $R \in \{K, \O, k \}$ given by $M^\lambda = \H x_\lambda$ where
$x_\lambda = \sum_{w \in \mathfrak{S}_\lambda} T_w$.
The $q$-Schur algebra $\Sch_R(k,n)$ is the endomorphism algebra
$\End_{\H}(\bigoplus_{\lambda \in \Lambda(k,n)} M^\lambda)$.
We are mostly interested in the $q$-Schur algebra
$\Sch_R(n,n)$ and certain of its subalgebras. Note however, that for $k \geq n$
there is a Morita equivalence between $\Sch_R(k,n)$ and
$\Sch_R(n)$.  We write $\Sch_R(n)$ or
$\Sch(n)$ for $\Sch_R(n,n)$.
The algebra $\Sch(n)$ is free over $R$ and has basis
$$\{ \Phi_{\lambda,\mu}^u | \lambda, \mu \in \Lambda(n), u \in
{\mathcal D_{\lambda,\mu}} \}$$ where ${\mathcal D_{\lambda,\mu}}$ is
the set of distinguished double coset representatives of $\mathfrak{S}_\lambda \backslash \mathfrak{S}_n
/\mathfrak{S}_\mu$ which are the unique elements of shortest length in
their double coset \cite{DJ}. The elements $\Phi_{\lambda,\lambda}^1$
for $\lambda \in \Lambda(n)$ are
orthogonal idempotents in $\Sch(n)$. If $\lambda = (1^n)$ then
multiplication by the idempotent $e=
\Phi_{\lambda,\lambda}^1$ is called the Schur functor and
$e \Sch(n) e \cong {\H}$. Furthermore, for each composition $\nu$,
we define a generalization of the Schur functor. That is,
we define an idempotent of $\Sch(n)$
$$e_\nu = \sum_{\lambda \in \Lambda(\nu)} \Phi_{\lambda,\lambda}^1.$$
The subalgebra $\Sch(\nu) = \Sch_R(\nu_1) \otimes \ldots \otimes
\Sch_R(\nu_h)$ naturally embeds
into $e_\nu \Sch(n) e_\nu$ (see for example \cite{BDK}, section 4.2 for details)
and we can
define the following functors
$$\begin{array}{ccccc}
\Sch(n)e_\nu \otimes_{\Sch(\nu)} - & : & \Sch(\nu) -\mod & \rightarrow &
\Sch(n)-\mod \\
&&&&\\
e_\nu. - & : &  \Sch(n)-\mod &\rightarrow & \Sch(\nu) -\mod.\\
\end{array}$$

The $q$-Schur algebra $\Sch_k(n)$ is quasi-hereditary, the system
of labels of the standard modules being partitions of $n$ ordered by
the dominance ordering. The general theory of quasi-hereditary
algebras produces a parametrisation
 of the simple $\Sch_k(n)$-modules $L(\lambda)$, $\lambda \in
 \Lambda^+(n)$. Furthermore, the set $\Lambda^+(n)$ also
 parametrizes the so-called standard modules $\Delta(\lambda)$ and the
 costandard modules
$\nabla(\lambda)$.

\bigskip

By  Takeuchi \cite{T} and  Brundan, Dipper, Kleshchev \cite{BDK}
the $q$-Schur algebra $\Sch_R(n)$ is Morita equivalent to the
quotient algebra $C_R(n)= R G /I $ where $I$ is the annihilator of the
permutation module $R (G/B)$ on the cosets $G/B$ of $B$ in $G$.
In fact the Morita equivalence is given
by a $RG$-$\Sch_R(n)$-bimodule $Q$ and its $C_R(n)$-linear dual (see
\cite[3.4g]{BDK} for the precise definition).
Furthermore, $C_\O(n) = \O G/I$ is isomorphic to $\O Gf$ (see \cite{S2} Thm 1
for a proof) and $fQ=Q$ thus $Q$ is also an $\O Gf$-module where $f$
is again the central idempotent of the sum of unipotent characters
defined in the beginning of section \ref{sec2}. Furthermore
$$\begin{array}{ccccc} Q \otimes_{\Sch(n)} - &: &\Sch(n)-\mod
  &\rightarrow &\O Gf-\mod \\
&&&&\\
Q' \otimes_{\O Gf} - &:& \O Gf-\mod& \rightarrow &\Sch(n)-\mod \\
\end{array}$$
are mutually inverse functors, where $Q'$ denotes the $\O
Gf$-linear dual of $Q$.

\bigskip

\subsection{Alvis-Curtis duality}

Originally Alvis-Curtis duality was introduced as a duality operator
on the ordinary character group of a finite group of Lie
type. Deligne and Lusztig defined it as the
Lefschetz character of a certain complex of the module category in
characteristic 0. Cabanes and
Rickard \cite{CR} formalized this definition by using coefficient
systems and extended it to
representations over any ring containing $p^{-1}$, where $q =
p^\alpha$. Although their
definition holds for all finite groups of Lie type we will present it
here only for $G=\GL_n(q)$.

\bigskip

Let $X_G$ be the complex associated to the  coefficient system
on the simplicial complex given by the
set of compositions of $n$ with the partial order $\preceq$. It sends
a composition $\lambda$ to the $RG$-$RG$-bimodule $RG e_{U_\lambda}
\otimes_{RL_\lambda} e_{U_\lambda} RG$ and the inclusion $\lambda
\prec \nu$ to the map
$$\begin{array}{cccc}
\alpha_{\lambda,\nu} : & RG e_{U_\lambda}
\otimes_{RL_\lambda} e_{U_\lambda} RG & \rightarrow & RG e_{U_\nu} 
\otimes_{RL_\nu} e_{U_\nu} RG\\
& x \otimes y  & \mapsto & x \otimes y.
\end{array}$$

The associated
cochain complex $X_G$, shifted and renumbered, such that it is concentrated in degrees $0$ to
$n-1$ has degree $i$th term
$$ X_G^i = \bigoplus_{\lambda \in \bar{\Lambda}(n-i,n)} RG
e_{U_\lambda} \otimes_{RL_\lambda} e_{U_\lambda} RG$$
for $0 \leq i \leq n-1$.  Its differential is given by
$$d^i = \sum_{\lambda \in \bar{\Lambda}(n-i,n)} \sum_{\stackrel{\nu
    \in \bar{\Lambda}(n-i-1,n)}{\lambda \prec \nu}} (-1)^{\lambda/\nu} \alpha_{\lambda,\nu}
$$
where $\lambda/\nu$ is the unique partial sum $\lambda_j^+$ of $\lambda$
such that there is no partial sum $\nu_k^+$ of $\nu$ such that
$\lambda _j^+ = \nu_k^+$.

\bigskip

In \cite{CR} Cabanes and Rickard showed
that tensoring with $X_G$
induces a derived self-equivalence of
the module category of $G$. Therefore we obtain an induced map in
the Grothendieck group of $G$, that is a bijection in characteristic
$0$,
$$ D_G(-)  =  \sum_{\lambda \in \Lambda(n)} (-1)^{|\lambda|} RG e_{U_\lambda} \otimes_{RL_\lambda}
e_{U_\lambda} (-). $$

\bigskip

For any standard Levi subgroup $L_\lambda = \GL_{\lambda_1}(q) \times \ldots
\times  \GL_{\lambda_a}(q)$ we define
$$X_{L_\lambda} = X_{\GL_{\lambda_1}(q)} \otimes \ldots \otimes
X_{\GL_{\lambda_a}(q)}.$$
In fact, it is easy to see that
$X_{L_\lambda}$ coincides with the complex defined for $L_\lambda$
when it is itself considered as a finite group of Lie type.

It then follows from \cite{CR} that in the Grothendieck group
Alvis-Curtis duality commutes with
Harish-Chandra
induction and restriction~:

\bigskip

\begin{Lemma}\cite[6.1]{CR}\label{DualityCommutes}
Let $L$ be a standard Levi subgroup of $G$. Then
$$R_L^G\circ D_L [N] = D_G\circ R_L^G [N]$$
$$\hcr^G_L\circ D_G [M] = D_L\circ \hcr^G_L [M]$$ for $[N]$ in
$K_0(RL)$ and $[M]$ in $K_0(RG)$ with $R \in \{K,\O, k\}$.
\end{Lemma}

\bigskip

\section{Modular Alvis-Curtis duality}

In the case of the general linear group $G$ the Alvis-Curtis dual of an irreducible
$KG$-module is well-understood (see below). In the modular case however,
little is known.
For example it is open what the
Alvis-Curtis dual of a simple $kG$-module is. Indeed it
turns out that this question can be transformed into a question on
decomposition matrices of general linear groups. More precisely we
will show that determining the Alvis-Curtis duality on simple
modules for general linear groups in non-describing characteristic
is equivalent to determining the
decomposition matrix of $G$.

\bigskip

  For the convenience of
the reader we recall the result on the simple $KG$-modules (see for example \cite{S}).

\bigskip

\begin{Proposition}\label{DGirred}
Let $S(s,\tilde{\lambda})$ be an irreducible $KG$-module.
Then in the Grothendieck group $K_0(KG)$
$$D_G([S(s,\tilde{\lambda})])= \varepsilon_s [S(s,\tilde{\lambda}'])$$
where $\tilde{\lambda}' = ({\lambda^{(1)}}',\dots,{\lambda^{(r)}}')$
 denotes the conjugate multipartition of $\tilde{\lambda}$ and
 $\varepsilon_s$ is defined as in section \ref{sec22}.
\end{Proposition}

\bigskip

Assume that the set of  partitions
$\Lambda^+(n)$ is ordered by a linear order $\leq$ refining the
dominance order, for example the lexicographic order.
Denote by ${\mathcal Z}_G$ the decomposition matrix of $G$ and let
${\mathcal Z}_u = (z_{\lambda,\nu})$ be the upper (quadratic) part of
${\mathcal Z}_G$ restricted to the unipotent block.
So we have that ${\mathcal Z}_u=(z_{\lambda,\nu})$ with
$z_{\lambda,\nu} = [S(1,\lambda) : L(1,\nu)]$ for partitions
$\lambda$, $\nu$ of $n$.

Denote by $A_G = (a_{\lambda,\nu})$ the matrix of the Alvis-Curtis
duality on the unipotent $kG$-modules, that is $a_{\lambda,\nu}$ is
the coefficient of $[L(1,\nu)]$ in $D_G[L(1,\lambda)]$. Finally let
$P$ be the permutation matrix given by the permutation on
$\Lambda^+(n)$ that sends $\lambda$ to $\lambda'$.

\bigskip

Then the decomposition matrix of $G$ determines the Alvis-Curtis
duality and the knowledge of the Alvis-Curtis duality for smaller
general linear groups
 determines the decomposition matrix of $G$. More precisely, we have

\bigskip

\begin{Theorem}\label{decompositionmatrixG}
Let $G=\GL_n(q)$. Then with the notations above

(i) $A_G$ is determined by the decomposition matrix of $G$ and
$A_G$ determines the unipotent part ${\mathcal Z}_u$
of the decomposition matrix
of $G$. Explicitly,
$A_G = {\mathcal Z}_u^{-1} \cdot P \cdot {\mathcal Z}_u$.

(ii) Alvis-Curtis duality on
$K_0(k\GL_{k}(q^{d}))$ for all $d$, $k$
such that $dk \leq n$
determines the whole of the
decomposition matrix ${\mathcal Z}_{G}$ of $G$.

%(iii) The decomposition matrices $\mathcal{Z}_{\GL_{k_i}(q^{d_i})}$ of
%$\GL_{k_i}(q^{d_i})$ determine the Alvis-Curtis
%duality on all irreducible $k\GL_{k_i}(q^{d_i})$-modules where
%$k_id_i\leq n$.
(iii) The decomposition matrix $\mathcal{Z}_{G}$ of $G$ determines
the Alvis-Curtis duality on all irreducible $kG$-modules.
\end{Theorem}

\bigskip

\begin{proof} {\it (i)}
Proposition \ref{DGirred} holds also for modular Specht modules and thus
$D_G([S(1,\lambda)])=[S(1,\lambda')]$ in $K_0(kG)$. This implies
${\mathcal Z}_u \cdot A_G = P\cdot {\mathcal Z}_u$. Since ${\mathcal Z}_u$ is lower unitriangular (see
\cite[16.3]{J1}) it is
invertible, so $A_G = {\mathcal Z}_u^{-1} P {\mathcal Z}_u$ and $A_G$ is determined by ${\mathcal Z}_G$.

On the other hand, since $A_G$ is invertible we can use the Bruhat decomposition
(see \cite[2.5.13]{Carter}) to write $A_G = U_1\cdot T\cdot R\cdot U_2$
in a unique way,
where $U_1,U_2$ are lower unitriangular, $T$ is a diagonal matrix and
$R$ is a permutation matrix such that $R\cdot
U_2\cdot R^{-1}$ is upper uni-triangular. We already know that ${\mathcal Z}_u$ and
therefore ${\mathcal Z}_u^{-1}$ are lower uni-triangular. Hence
$T=1$, $R=P$ and to complete the proof that $U_2={\mathcal Z}_u$ it remains to show
that $P\mathcal{Z}_uP^{-1}$ is upper unitriangular. By \cite{J1} Cor. 16.3
 we know that $(P{\mathcal Z}_uP^{-1})_{\lambda\mu}=d_{\lambda'\mu'}=0$ unless
$\lambda'\unrhd \mu'$. But this is equivalent to $\lambda \unlhd \mu$ and
therefore $P{\mathcal Z}_uP^{-1}$ is upper unitriangular. So ${\mathcal Z}_u$ is determined by $A_G$ as
the unique lower unitriangular matrix in the Bruhat decomposition of $A_G$.

{\it (ii)}
By the work of Dipper and James \cite{DJ} we know that the upper
(unitriangular) parts of the decomposition matrices of the unipotent blocks of
the groups ${\rm GL}_{k}(q^{d})$ for $kd\leq n$ determine the
decomposition matrix of $G$. So {\it (ii)}
is an easy consequence of part {\it (i)}.

{\it (iii)} The decomposition matrix $\mathcal{Z}_G$ has an upper (quadratic)
part that is lower unitriangular (see \cite{DipperOnDec}). Therefore every
irreducible $kG$-module in $K_0(kG)$ can be written as integral linear
combination of Specht modules. By Proposition
\ref{DGirred} we know the duality on Specht modules. Therefore as in (i) the
decomposion matrix determines the duality operator.
\end{proof}

We would like to point out that the converse to Theorem
\ref{decompositionmatrixG}(iii) is not true in general since the upper
(quadratic) part of the decomposition matrix does not determine the whole of
the decomposition matrix. To calculate the whole decomposition matrix of
$\GL_n(q)$ one needs to know all the quadratic parts of all
decomposition matrices for
all groups $\GL_k(q^d)$ where $dk\leq n$. In terms of Alvis-Curtis duality
this equals the knowledge of the duality operator on all these groups as in
Theorem \ref{decompositionmatrixG}(ii).

\bigskip

\section{Modular Alvis-Curtis duality on Hecke algebras}

Let $\H = \H_k$.
Recall from \cite{LS} the complex $X_\H$ of $\H$-$\H$-bimodules
defined by
$$ X_\H^i = \bigoplus_{\lambda \in \bar{\Lambda}(n-i,n)} \H
\otimes_{\H_\lambda} \H$$
for $0 \leq i \leq n-1$ and with differential
$$d^i = \sum_{\lambda \in \bar{\Lambda}(n-i,n)} \sum_{\stackrel{\nu
    \in \bar{\Lambda}(n-i-1,n)}{\lambda \prec \nu}} (-1)^{\lambda/\nu} \vartheta_{\lambda,\nu}
$$
where $\vartheta_{\lambda,\nu} : \H \otimes_{\H_\lambda} \H \rightarrow
\H \otimes_{\H_\nu} \H$ is the canonical surjection.

\bigskip

Note that this complex can be constructed as the chain complex
associated to a coefficient system on the same simplicial complex as
for $X_G$.

\bigskip

Furthermore, it is shown in \cite{LS} that tensoring with $\H$ induces
a derived self-equivalence of the module category of $\H$, that the
cohomology of $\H$ is concentrated in degree $0$ and that it is
isomorphic to $_\alpha\H$ as an $\H$-$\H$-bimodule.

\bigskip

Therefore $X_\H$ induces a duality operator $D_\H$ in the Grothendieck
group $K_0(\H)$ given by
\begin{equation}\label{Formelreferenz}
D_\H ([V]) = \sum_{\lambda \in \Lambda(n)} (-1)^{|\lambda|}
\Ind_{\mathcal{H}_\lambda}^\mathcal{H}\circ
\Res_{\mathcal{H}_\lambda}^\mathcal{H} [V] = [_\alpha\H \otimes_{\H}
V] = [_\alpha V]
\end{equation}
for $[V] \in K_0(\H)$.

\bigskip

In \cite{S2} it was shown that the complex $X_\H$ is isomorphic to a
quotient of the complex $X_G$, namely to $e_U X_G f e_U$. This
relation on the level of complexes is reflected on the level of the
Grothendieck groups in the following form. Denote by $A_\H=(h_{\lambda,
  \nu})$ the matrix whose entries are given by the coefficients of
$[L^\nu]$ in a decomposition of $D_\H([L^\lambda])$ for $\lambda, \nu$
$e$-regular partitions of $n$.

\bigskip

\begin{Theorem}\label{Bernd1}
The matrix $A_G=(a_{\lambda, \nu})$ giving the Alvis-Curtis duality completely
determines the matrix $A_\H$. More precisely, for $\lambda, \nu$
$e$-regular partitions of $n$ we have
$$
a_{\lambda,\nu} = h_{\lambda,\nu} = \left\{
\begin{array}{ll}
1 & \mbox{if} \;\;  \nu = m(\lambda) \\
0 & \mbox{otherwise} \\
\end{array} \right.
$$
\end{Theorem}

\bigskip

For the proof of theorem~\ref{Bernd1} we need the following lemma,
which is a direct consequence of \cite{BDK}, Corollary 3.2(g) that:

\begin{Lemma}\label{Bernd2}
Let $M$ denote the permutation module $M= \hci^G_T(k)$. Then
the following diagram commutes:
\begin{equation*}
\begin{CD}
\mathcal{H}-mod @>{M\otimes_\mathcal{H}\_}>> kG-mod \\
@V{\Ind^\mathcal{H}_{\mathcal{H}_\lambda}\circ
\Res_{\mathcal{H}_\lambda}^\mathcal{H}}VV
@VV{\hci^G_{L_\lambda}\circ\hcr^G_{L_\lambda}}V  \\
\mathcal{H}-mod @>{M\otimes_\mathcal{H}\_}>> kG-mod
\end{CD}
\end{equation*}
\end{Lemma}
\begin{proof}
This is Cor. 3.2(g) of \cite{BDK}.
\end{proof}

\bigskip

{\it Proof of Theorem~\ref{Bernd1}.} Let $\beta$ denote the functor $M\otimes_\mathcal{H}\_$. Since
  $\beta$ is exact Lemma \ref{Bernd2} gives us $D_G\circ\beta = \beta\circ
  D_\mathcal{H}$ on the level of Grothendieck groups. And since
  $\beta(L^\lambda)\cong L(1,\lambda)$ we have
\begin{align*}
\sum_\mu a_{\lambda,\mu} [L(1,\mu)]&=D_G\circ\beta([L^\lambda]) \\
&=\beta\circ
D_{\mathcal{H}}([L^\lambda])=\beta(\sum_\mu h_{\lambda,\mu}[L^\mu]=\sum_{\mu}h_{\lambda,\mu}[L(1,\mu)]
\end{align*}
which proves the first equality.
The second equality follows directly from formula~(\ref{Formelreferenz}).

\hfill$\Box$

\bigskip

\section{Modular Alvis-Curtis duality on $q$-Schur algebras}

We define a duality operator on the Grothendieck group of
$\Sch(n)$-modules induced by a derived self-equivalence of the module
category given by tensoring with a complex.
This duality operator is closely linked to the Alvis-Curtis
duality operator on the level of the Grothendieck group as well as
on the level of the derived category.
Furthermore, we show how this complex relates to the complex $X_\H$
for  the  Hecke algebra defined in the previous section.

\bigskip

Let $\Sch(n)=\Sch_R(n)$ be the $q$-Schur algebra defined over $R \in \{\O,
k\}$.
Define the cochain complex $X_{\mathcal S}$ of
$\Sch(n)$-$\Sch(n)$-bimodules as the cochain complex associated to
the coefficient system on the simplicial complex of compositions of
$n$ with partial order $\preceq$. A composition $\lambda$ of $n$
is sent to the $\Sch(n)$-$\Sch(n)$-bimodule $\Sch(n)e_\lambda \otimes_{S(\lambda)}
e_\lambda \Sch(n)$ and the inclusion $\lambda \prec \nu $ is sent to
the map
$$ \begin{array}{ccccc} \gamma_{\lambda,\nu} &:& \Sch(n)e_\lambda \otimes_{S(\lambda)}
e_\lambda \Sch(n) & \rightarrow & \Sch(n)e_\nu \otimes_{S(\nu)}
e_\nu \Sch(n)\\
&& x\otimes y & \mapsto & x\otimes y.\\
\end{array}$$
Note that for $\lambda$, $\nu$ as above, $e_\nu e_\lambda = e_\lambda
e_\nu = e_\nu$ thus the maps are well-defined. Hence
after shifting and renumbering the complex $X_{\mathcal S}$ is concentrated
in degrees $0$ to $n-1$ and for $0 \leq i \leq n-1$
$$X_{\mathcal S}^i = \bigoplus_{\lambda \in \bar{\Lambda(n-i,n)}}
\Sch(n)e_\lambda \otimes_{\Sch(\lambda)} e_\lambda \Sch(n).$$
The differential is given by
$$d^i = \sum_{\lambda \in \bar{\Lambda}(n-i,n)} \sum_{\stackrel{\nu
    \in \bar{\Lambda}(n-i-1,n)}{\lambda \prec \nu}} (-1)^{\lambda/\nu} \gamma_{\lambda,\nu}.
$$

\bigskip

Then  $X_\Sch$ induces a derived self-equivalence and a duality
operation in the Grothendieck group of the module category of $\Sch(n)$.

\bigskip

\begin{Theorem}\label{qSchur}
The functor $X_{\mathcal S} \otimes_{\Sch(n)} - : {\mathcal D}^\flat(\Sch(n)) \to {\mathcal
  D}^\flat(\Sch(n))$ is an equivalence of categories inducing a
  duality operator $D_{\Sch}$ of the Grothendieck group
$K_0(\Sch_k(n))$ with the following properties

(i) $D_{\Sch}(\Delta(\lambda)) = \Delta(\lambda')$ for all
standard modules $\Delta(\lambda)$.

(ii) $D_{\Sch}(\nabla(\lambda)) =
\nabla(\lambda')$ for all
costandard modules $\nabla(\lambda)$.

(iii) $D_{\Sch}(L(\lambda)) = \sum_{\mu \in \Lambda^+(n)}
a_{\lambda'\mu'} L(\mu)$, where $A_G=(a_{\lambda\mu})$ is the matrix
given by the Alvis-Curtis duality.
 \end{Theorem}

\bigskip

\begin{Theorem}
There is a split monomorphism of complexes of $\H_R$-$\H_R$-bimodules
from $X_\H$ into $ eX_\Sch e$.
\end{Theorem}

\bigskip

\begin{proof}
Let $\H = \H_R$ and recall that $e \Sch(n) e \cong \H$. By \cite[4.6(5)]{D} the idempotent $e$
also satisfies $e \Sch(\lambda) e \cong \H_\lambda$ for all
compositions $\lambda$ of $n$. Therefore the terms of the complex
$X_\H$ are isomorphic to direct sums of terms of the form
$e\Sch(n) e \otimes_{e\Sch(\lambda)e} e
\Sch(n) e$. The map defined by
$$
\begin{array} {ccccc}
 &  & e\Sch(n) e_\lambda \otimes_{\Sch(\lambda)} e_\lambda \Sch(n) e
& \rightarrow & e\Sch(n) e \otimes_{e\Sch(\lambda)e} e
\Sch(n) e \\
&& x \otimes y & \mapsto & xe \otimes ey \\
\end{array}
$$
induces then a morphism of complexes. It is a splitting map for the
injection inducing a monomorphism of complexes
$$
\begin{array} {ccc}
 e\Sch(n) e \otimes_{e\Sch(\lambda)e} e
\Sch(n) e & \rightarrow &e\Sch(n) e_\lambda \otimes_{\Sch(\lambda)}
e_\lambda \Sch(n) e \\
x \otimes y & \mapsto & x \otimes y. \\
\end{array}
$$
\end{proof}

Before proving theorem \ref{qSchur} let us give some preliminary results.

\bigskip

By \cite[4.2c]{BDK} the functor $\Sch(n)e_\lambda
\otimes_{\Sch(\lambda)} - $ corresponds to Harish-Chandra
induction in the general linear group. More precisely, if $M$ is an
$\Sch(\lambda)$-$\Sch(n)$-bimodule then there is an ismorphism
of $C(n)$-$\Sch(n)$-bimodules
$$
R_{L_\lambda }^G(Q_\lambda \otimes_{\Sch(\lambda)} M) \cong Q
\otimes_{\Sch(n)} \Sch(n) e_\lambda \otimes_{\Sch(\lambda)} M \cong  Q e_\lambda \otimes_{\Sch(\lambda)} M.
 $$
Here $Q_\lambda$ is the $\Sch(\lambda)$-$R L_\lambda$ bimodule inducing
a Morita equivalence between $\Sch(\lambda)$ and $C(\lambda)$.
That is the following diagram commutes
\begin{equation*}
\begin{CD}
\Sch(\lambda)-\mod -\Sch(n) @>{Q_\lambda \otimes_{\Sch(\lambda)} - }>>
C(\lambda) - \mod - \Sch(n) \\
@V{\Sch(n)e_\lambda \otimes_{\Sch(\lambda)} - }VV @VV{R^G_{L_\lambda}}V \\
\Sch(n) - \mod - \Sch(n) @>{Q \otimes_{\Sch(n)} - }>>  C(n) - \mod - \Sch(n)
\end{CD}
\end{equation*}

We will show that for $q$-Schur algebras multiplying with the idempotent
$e_\lambda$
corresponds to Harish-Chandra restriction. More precisely,
we will show that we have a commutative diagram

\begin{equation*}
\begin{CD}
\Sch(n) - \mod - \Sch(n) @>{Q \otimes_{\Sch(n)} - }>>  C(n) - \mod - \Sch(n)\\
@V{e_\lambda .  }VV @VV{^{*}R^G_{L_\lambda}}V \\
\Sch(\lambda)-\mod -\Sch(n) @>{Q_\lambda \otimes_{\Sch(\lambda)} - }>>
C(\lambda) - \mod - \Sch(n).
\end{CD}
\end{equation*}

That is the following holds

\bigskip

\begin{Lemma}
For any $\Sch(n)$-$\Sch(n)$-bimodule $N$, there is an
$C(\lambda)$-$\Sch(n)$-bimodule isomorphism
$$\hcr^G_{L_\lambda} Q \otimes_{\Sch(n)} N \cong Q_\lambda \otimes_{\Sch(\lambda)}
e_{\lambda} N.$$
\end{Lemma}

\bigskip

\begin{proof}
The Morita equivalence between the subalgebras $\Sch(\lambda)$ and
$C(\lambda)$ induces an isomorphism of $C(\lambda)$-$\Sch(n)$-bimodules
$$\hcr^G_{L_\lambda} Q \otimes_{\Sch(n)} N \cong
Q_\lambda
\otimes_{\Sch(\lambda)} \Hom_{C(\lambda)}(Q_\lambda, C_{\lambda}) \otimes_{C(\lambda)}
{\hcr^G_{L_\lambda}} Q \otimes_{\Sch(n)} N.
$$
As $Q_\lambda$ is a projective $C(\lambda)$-module by
\cite[20.10]{AF} the term on the right hand side is isomorphic to
$
Q_\lambda
\otimes_{\Sch(\lambda)} {\rm Hom}_{C(\lambda)} (Q_\lambda,\hcr^G_{L_\lambda}
Q \otimes_{\Sch(n)} N ).$ By \cite[4.2]{BDK}
$R^G_{L_\lambda}Q_\lambda$ and $Qe_\lambda$
are isomorphic as
$C(n)$-$\Sch(\lambda)$-bimodules. Therefore adjunction of
Harish-Chandra induction and restriction
gives an isomorphism
$$
Q_\lambda
\otimes_{\Sch(\lambda)} {\rm Hom}_{C(\lambda)} (Q_\lambda,\hcr^G_{L_\lambda}
Q \otimes_{\Sch(n)} N ) \cong Q_\lambda
\otimes_{\Sch(\lambda)} \Hom_{C(n)}(Qe_\lambda, Q \otimes_{\Sch(n)}
N).$$
Now applying \cite[20.10]{AF} again, we obtain
$$ \Hom_{C(n)}(Qe_\lambda, Q \otimes_{\Sch(n)} N) \cong
\Hom_{C(n)}(Qe_\lambda, Q) \otimes_{\Sch(n)} N$$ where
$\Hom_{C(n)}(Qe_\lambda, Q) \cong e_\lambda \End_{C(n)}(Q) \cong
e_\lambda \Sch(n)$. The result follows.
\end{proof}

\bigskip

In particular, the two squares above still commute when we replace
$C_\O(n)$ by $\O Gf$ and $C_\O(\lambda)$ by $\O L_\lambda f_\lambda$
as well as when the horizontal arrows are reversed, that is when $Q
\otimes_{\Sch(n)} - $ is replaced by $Q' \otimes_{\O G f} - $ and
$Q_\lambda \otimes_{\Sch(\lambda)} - $ is replaced by $Q'_\lambda \otimes_{\O L_\lambda f_\lambda} - $.

\bigskip

In order to show that tensoring with $X_\Sch$ induces a derived equivalence, we
first show that $X_\Sch$ is isomorphic to the following complex $Y$.
Let
 $F =
Q' \otimes_{\O Gf} -  \otimes_{\O Gf} Q$ be the bimodule functor
induced by the Morita equivalence
 between $\O Gf$ and $\Sch_\O(n)$ and set
$$Y = F(fX_Gf).$$
Then $Y$ is a complex of $\Sch_\O(n)$-$\Sch_\O(n)$-bimodules with degree
$i$th term
$$
Y^i = \bigoplus_{\lambda \in \bar{\Lambda}(n-i,n)} Q' e_{U_\lambda}
\otimes_{\O L_\lambda f_\lambda} e_{U_\lambda} Q$$
for $0 \leq i \leq n-1$ and with differential
$$d^i = \sum_{\lambda \in \bar{\Lambda}(n-i,n)} \sum_{\stackrel{\nu
    \in \bar{\Lambda}(n-i-1,n)}{\lambda \prec \nu}} (-1)^{\lambda/\nu} F(\alpha_{\lambda,\nu}).
 $$

\bigskip

\begin{Proposition}
The complexes $X_\Sch$ and $Y$ are isomorphic as complexes of
$\Sch_\O(n)$-$\Sch_\O(n)$-bimodules.
\end{Proposition}

\bigskip

\begin{proof}
We start by showing the terms in each degree are isomorphic
bimodules. Let $\lambda$ be a composition of $n$, then
$\Sch(n) e_\lambda \otimes_{\Sch(\lambda)} e_\lambda \Sch(n) =
\Sch(n) e_\lambda \otimes_{\Sch(\lambda)} Q'_\lambda
\otimes_{\Sch(\lambda)} Q_\lambda
\otimes_{\Sch(\lambda)} e_\lambda \Sch(n)$. By the preceding lemma
$Q_\lambda
\otimes_{\Sch(\lambda)} e_\lambda \Sch(n) \cong e_{U_\lambda} Q$. In
 a similar way $\Sch(n) e_\lambda \otimes_{\Sch(\lambda)} Q'_\lambda
otimes_{\Sch(\lambda)} e_{U_\lambda} Q$ is isomprphic to
$Q' \otimes_{\O Gf} \O Gf e_{U_\lambda}
\otimes_{\O L_\lambda f_\lambda} e_{U_\lambda} Q \cong Q' e_{U_\lambda}
\otimes_{\O L_\lambda f_\lambda} e_{U_\lambda} Q$. Then as all isomorphisms
involved are functorial, they clearly commute with the differentials.
\end{proof}

\bigskip

{\it Proof of theorem \ref{qSchur}.}
The functor $$Q' \otimes_{\O Gf} - \otimes_{\O Gf} Q \; :\;\;  \O Gf - \mod -\O
Gf \;\; \longrightarrow  \;\; \Sch(n) -\mod - \Sch(n)$$ is an equivalence of bimodule
categories. The complex  $Y$ is isomorphic to
$Q' \otimes_{\O Gf} fX_Gf \otimes_{\O Gf} Q$ and by
\cite{S2} the functor $f X_G f \otimes_{\O Gf} - $ induces a self-derived
equivalence of the module category of $\O Gf$.  Therefore the functor
$Y \otimes_{\Sch_\O(n)} -$ induces a self-equivalence of ${\mathcal D}^\flat (\Sch_\O(n))$.
By \cite[2.2]{R} the complex $k \otimes_{\O} Y = Y_k$ induces a
self-derived equivalence of the module category of $\Sch_k(n)$.

\bigskip

Denote by $D_\Sch$ the induced map in the Grothendieck group
$K_0(\Sch_R(n))$ for $R \in \{K, k\}$.
ia the Morita equivalence between $\Sch(n)$ and $C(n)$,
he standard module $\Delta(\lambda)$ corresponds to
the Specht module $S(1,\lambda')$ labelled by the conjugate
partition and the costandard module $\nabla(\lambda)$ corresponds to
the dual Specht module $S(1,\lambda)^*$ (see \cite{T} or \cite{BDK}).
It follows then from proposition \ref{DGirred} and the above
that
$D_\Sch$ sends $\Delta(\lambda)$ to $\Delta(\lambda')$ and
$\nabla(\lambda)$ to $\nabla(\lambda)'$ in $K_0(\Sch_R(n))$.
Furthermore, $L(\lambda)$ is the head of $\Delta(\lambda)$ and $P
{\mathcal Z}_u P^{-1}$ is the decomposition matrix for
$\Sch_k(n)$. Thus part {\it (iii)} of the theorem follows.

\hfill$\Box$

\bigskip

\section{Decomposition matrices}

In this last section we recapitulate the information we have
obtained on the different decomposition matrices. Recall the notation
${\mathcal Z}_G$ and ${\mathcal Z}_u$ for the decomposition matrix of
$G$ and its unipotent part. Set $\Z_u^e$ for the matrix given by the
rows and columns of $\Z_u$ indexed by $e$-regular partitions of
$n$. Denote by $\Z_\H$ the decomposition matrix
of $\H$, that is the matrix whose coefficients are given by
$[S^\lambda : D^\nu]$ for $\lambda$ a partition of $n$ and $\nu$ an
$e$-regular partition of $n$. Finally denote by $\Z_\Sch$ the
decomposition matrix of the $q$-Schur algebra $\Sch(n)$ with
coefficients given by $[\Delta(\lambda) : L(\nu)]$ for partitions
$\lambda$, $\nu$ of $n$. Recall that $\Z_\Sch = P \Z_u P^{-1}$.

Then combining the previous results we have the following relations

\bigskip

\begin{Theorem} The following matrix equations hold

(i) $A_G = \Z_u^{-1} P \Z_u$

(ii) $\Z_u^e A_\H = P \Z_u^e $

(iii) $A_\Sch =  \Z_\Sch^{-1} P \Z_\Sch = P^{-1} A_G P. $
\end{Theorem}

\bibliographystyle{acm}

\end{document}